\documentclass[12pt]{article}

\usepackage{amssymb,amsfonts,amsthm}
\usepackage[centertags]{amsmath}
\usepackage[english]{babel}

\usepackage{amssymb,  mathrsfs}
\usepackage{graphicx}

\usepackage{lineno}

\usepackage{color}

\newtheorem{theorem}{Theorem}
\newtheorem{lemma}[theorem]{Lemma}

\newtheorem{observation}[theorem]{Observation}
\newtheorem{corollary}[theorem]{Corollary}

\newtheorem{them}{Theorem}

\newtheorem{prob}[them]{Problem}

\newtheorem{definition}[theorem]{Definition}

\newtheorem{remark}[theorem]{Remark}

\title{ Domination related parameters in the generalized lexicographic product of graphs} 
\author{ Vladimir Samodivkin \\
Department of Mathematics \\
University of Architecture Civil Engineering and Geodesy\\
Hristo Smirnenski 1 Blv., 1046 Sofia, Bulgaria,\\
 \texttt{vl.samodivkin@gmail.com}
 }

\date{}

\begin{document}

\maketitle


\begin{abstract}

In this paper we  begin an exploration of     several domination-related parameters 
(among which  are the total,  restrained, total restrained, paired, outer connected and
 total outer connected  domination numbers)  in the generalized lexicographic product (GLP for short) of graphs.  
We prove that for each GLP of graphs there exist several equality chains  containing  these parameters.
Some known results on standard lexicographic product of two graphs are  generalized or/and extended. 
We also obtain results on well $\mu$-dominated GLP of graphs, where 
$\mu$ stands for any of  the above mentioned domination parameters. 
In particular, we present a characterization of well $\mu$-dominated GLP of graphs 
in the cases when $\mu$ is the domination number or the total domination number. 
\end{abstract}

generalized lexicographic product; equality chains.


\newpage

\section{Introduction}

One of the fastest growing areas within graph theory is the study of domination.
Many variants of the basic concepts of domination have appeared in the literature. 
 We refer to  \cite{hhs1} for a survey of the area. 
As many other graph invariants, domination has been studied on different graph products. 
Several papers have been published in the last fifteen years concerning
various types domination in the lexicographic product of two graphs,  including
domination (Nowakowski and Rall \cite{nr}, \u{S}umenjak et al.  \cite{spt} and G\"{o}z\"{u}pek et al. \cite{GHM1}),
total and restrained domination (Zhang et al. \cite{zlm}), 
Roman domination (\u{S}umenjak et al. \cite{spt}), 
rainbow domination (\u{S}umenjak et al.  \cite{srt}), 
super domination (Dettlaff et al. \cite{dlrz}) and double domination (Cabrera Mart\'inez et al.\cite{ccr}).  
However, to the best knowledge of the author, 
there are no studies related to domination in graphs representable as generalized lexicographic products. 
This fact motivates us to begin an exploration of several domination-related parameters 
(among which  are the total,  restrained, total restrained, outer connected 
  and total outer connected  domination numbers)  
	in the generalized lexicographic product of graphs.

We give basic terminologies and notations in the rest of this section. 
All graphs in this paper will be finite, simple, and undirected. 
We use \cite{hhs1} as a reference for terminology and notation which are not explicitly defined here. 

In a graph $G$, for a subset $S \subseteq V (G)$ the {\em subgraph induced} by $S$ is the graph
$\left\langle S \right\rangle$ with vertex set $S$ and edge set $\{xy \in E(G) \mid x, y \in S\}$. 
The {\em complement} $\overline{G}$ of $G$ is the graph whose
vertex set is $V (G)$ and whose edges are the pairs of nonadjacent vertices of $G$.
We write $K_n$ for the {\em complete graph} of order $n$ and $P_n$ for the  {\em path} 
on $n$ vertrices. Let $C_m$ denote the {\em cycle} of length $m$.
 For any vertex $x$ of a graph $G$, $N_G(x)$ denotes the set of all neighbors 
of $x$ in $G$, $N_G[x] = N_G(x) \cup \{x\}$  and the degree of $x$ is $deg_G(x) = |N_G(x)|$. 
The {\em minimum} and {\em maximum} degrees
 of a graph $G$ are denoted by $\delta(G)$ and $\Delta(G)$, respectively.
For a subset $S \subseteq V (G)$, let $N_G[S] = \cup_{v \in S}N_G[v]$. 
The distance between vertices $x$ and $y$ of a graph $G$ is denoted by $dist_G(x,y)$. 
An {\em isomorphism} of graphs $G$ and $H$ is a bijection 
$f \colon V(G)\to V(H)$ such that any two vertices $u$ and $v$ of $G$
 are adjacent in $G$ if and only if $f(u)$ and $f(v)$ are adjacent in $H$. 
If an isomorphism exists between two graphs, then the graphs are called isomorphic
 and denoted as $ G\simeq H$. We use the notation $[k]$  for $\{1,2,..,k\}$.

Let $G$ be a  graph with vertex set
 $V(G) = \{\textbf{1}, \textbf{2},..,\textbf{n}\}$ and   let 
 $\Phi = (F_1,F_2,..,F_n)$ be  an ordered $n$-tuple of  paired disjoint graphs. 
Denote by $G[\Phi]$ the graph with  vertex set $\cup_{i=1}^n V(F_i)$ 
and edge set defined as follows: (a) $F_1, F_2,.., F_n$ are induced subgraphs of $G[\Phi]$,  
 and  (b) if $x \in V(F_i)$, $y \in V(F_j)$, $i,j \in [n]$ and $i \not= j$, then 
$xy \in E(G[\Phi])$ if and only if \textbf{ij} $\in E(G)$. 
A graph  $G[\Phi]$ is called  the  {\em generalized lexicographic product}  of $G$ and $\Phi$.   
If  $F_i \simeq F$ for every $i=1,2,..,n$, then  $G[\Phi]$
 becomes the standard lexicographic product  $G[F]$. 
Each subset $U = \{u_1,u_2,..,u_n\} \subseteq V(G[\Phi])$ such that 
                         $u_i \in V(F_i)$, for every $i \in [n]$, is called a $G$-layer.
From the definition of $G[\Phi]$ it immediately follow:
\begin{itemize}
\item[(A)] (folklore) $G[\Phi] \simeq G$  if and only if  $G[\Phi] = G[K_1]$. 
                         $G[F] \simeq F$  if and only if $G \simeq  K_1$.
                        If $G$ has at least  two vertices, then $G[\Phi]$ is connected if and only if $G$ is connected. 
												If $G$ is edgeless, then $G[\Phi] = \cup_{i=1}^nF_i$.
                        For any $G$-layer $U = \{u_1,u_2,..,u_n\}$               
                        the bijection  $f \colon V(G)\to U$ defined by  $f(\textbf{i}) = u_i \in V(F_i)$ is an 
										    isomorphism between $G$ and $\left\langle  U \right\rangle$. 
												For any $x \in V(F_i)$ and $y \in V(F_j)$, $i \not= j$, is fulfilled  
												$dist_{G[\Phi]}(x,y) = dist_G(\textbf{i},\textbf{j})$. 
\end{itemize} 
 The equality $dist_{G[\Phi]}(x,y) = dist_G(\textbf{i},\textbf{j})$ 
will be used in the sequel without specific references.

 Since for any domination-related parameter $\mu$, which we consider in this work,
and for any two disjoint graphs $G_1$ and $G_2$  is fulfilled 
$\mu(G_1 \cup G_2) = \mu(G_1) + \mu(G_2)$ 
(when it is known that at least one of the left or right sides of this equality exists), 
we restrict our attention only on 
connected generalized lexicographic products.  Therefore, 
 in what follows when a graph $G[\Phi]$ is under consideration 
we assume that $G$ is a connected graph of order $n \geq 2$.  
 Unless otherwise stated, we also  assume that  always $\Phi = (F_1,F_2,..,F_n)$. 

Let $\mathcal{I}$ denote the set of all mutually nonisomorphic graphs. 
	A {\em graph property} is any non-empty subset of $\mathcal{I} $. 
	We say that a {\em graph $G$ has property} $\mathcal{P}$ whenever 
	there exists a graph $H \in \mathcal{P}$ which is isomorphic to $G$. 
For example we list some graph properties: 

\medskip

$\bullet$    $\mathcal{T} = \{H \in \mathcal{I}$ : $\delta(H) \geq 1\}$;

$\bullet$
   $\mathcal{F} = \{H \in \mathcal{I}$ : $H$ is a forest$\}$;
   
$\bullet$      $\mathcal{M} = \{H \in \mathcal{I}$ : $H$ has a perfect matching $\}$;

$\bullet$    $\mathcal{S}_k = \{H \in \mathcal{I}$ :    $\Delta(G) \leq k \}$, $k \geq 0$.

$\bullet$  $\mathcal{C} = \{H \in \mathcal{I}$ :   $H$ is connected$\}$.
  
\medskip  

Any  set $S \subseteq V(G)$ such that  $\left\langle S  \right\rangle$
  possesses the property $\mathcal{A} \subseteq \mathcal{I}$ and 
	$\left\langle V(G)-S  \right\rangle$ possesses the property $\mathcal{B} \subseteq \mathcal{I}$
		is called an $(\mathcal{A}, \mathcal{B})$-{\em set}. 
		A {\em dominating set} for a graph $G$ is a set of vertices $D \subseteq V(G)$ 
		such that every vertex of $G$ is either in $D$ or is adjacent to an element of $D$.
A dominating $(\mathcal{A}, \mathcal{B})$-set  $S$ of a graph $G$  is 
a {\em minimal  dominating $(\mathcal{A}, \mathcal{B})$-set} 
if no set $S^{'} \subsetneq S$ is a dominating $(\mathcal{A}, \mathcal{B})$-set. 
The set of all minimal dominating $(\mathcal{A}, \mathcal{B})$-sets of a graph $G$
 is  denoted by $MD_{(\mathcal{A}, \mathcal{B})} (G)$.
The {\em domination number with respect to the pair} $(\mathcal{A}, \mathcal{B})$, 
denoted by $\gamma_{(\mathcal{A}, \mathcal{B})} (G)$, 
is the smallest cardinality of a dominating $(\mathcal{A}, \mathcal{B})$-set of $G$. 
 The {\em upper  domination number with respect to the pair} $(\mathcal{A}, \mathcal{B})$, 
denoted by $\Gamma_{(\mathcal{A}, \mathcal{B})} (G)$,
 is the maximum cardinality of a minimal dominating $(\mathcal{A}, \mathcal{B})$-set of $G$. 
A $\gamma_{(\mathcal{A}, \mathcal{B})}$(resp., $\Gamma_{(\mathcal{A}, \mathcal{B})}$)-{\em set} 
of a graph $G$ is every set in  $MD_{(\mathcal{A}, \mathcal{B})} (G)$ 
having cardinality $\gamma_{(\mathcal{A}, \mathcal{B})} (G)$ (resp., $\Gamma_{(\mathcal{A}, \mathcal{B})}(G)$).
Note that: 
\begin{itemize}
\item[\rm (a)] $\gamma_{(\mathcal{I}, \mathcal{I})} (G)$ and  $\Gamma_{(\mathcal{I}, \mathcal{I})} (G)$  are known as 
                            the  domination and upper domination numbers $\gamma(G)$ and $\Gamma(G)$ of $G$, respectively,
\item [\rm (b)] $\gamma_{(\mathcal{S}_0, \mathcal{I})} (G)$ and  $\Gamma_{(\mathcal{S}_0, \mathcal{I})} (G)$ are known as 
                            the independent domination number $i(G)$ and the independence number $\beta_0(G)$, 
\item [\rm (c)] 	$\gamma_{(\mathcal{T}, \mathcal{I})} (G)$ and  $\Gamma_{(\mathcal{T}, \mathcal{I})} (G)$ are  known as 
                            the total domination and upper total domination numbers $\gamma_t(G)$ and $\Gamma_t(G)$ (\cite{cdh}), 													
	\item [\rm (d)] $\gamma_{(\mathcal{I}, \mathcal{T})} (G)$ and  $\Gamma_{(\mathcal{I}, \mathcal{T})} (G)$ are  known as 
                            the restrained domination and upper restrained domination numbers $\gamma_r(G)$ and $\Gamma_r(G)$ (\cite{t}),	
	\item [\rm (e)] $\gamma_{(\mathcal{T}, \mathcal{T})} (G)$ and  $\Gamma_{(\mathcal{T}, \mathcal{T})} (G)$ are known as 
                     the total restrained domination and upper total restrained domination numbers $\gamma_{tr}(G)$ and $\Gamma_{tr}(G)$ (\cite{clm}).	
	\item [\rm (f)] $\gamma_{(\mathcal{I}, \mathcal{C})} (G)$ and  $\Gamma_{(\mathcal{I}, \mathcal{C})} (G)$ are known as 
                            the  outer-connected domination and upper  outer-connected domination numbers 
														$\gamma^{oc}(G)$ and $\Gamma^{oc}(G)$ (\cite{cy}),		
	\item [\rm (g)] $\gamma_{(\mathcal{T}, \mathcal{C})} (G)$ and  $\Gamma_{(\mathcal{T}, \mathcal{C})} (G)$ are known as 
                            the  total outer-connected domination and upper  total outer-connected domination numbers 
														$\gamma^{oc}_t(G)$ and $\Gamma^{oc}_t(G)$ (\cite{cyt}), 
	\item [\rm (h)] $\gamma_{(\mathcal{M}, \mathcal{I})} (G)$ and  $\Gamma_{(\mathcal{M}, \mathcal{I})} (G)$ are known as 
                            the paired domination and upper paired domination numbers $\gamma_p(G)$ and $\Gamma_p(G)$ (\cite{hs}).															
\end{itemize}
The following inequalities are folklore: $\gamma(G) \leq \gamma_t(G) \leq \min\{\gamma_p(G), \gamma_{tr}(G), \gamma_t^{oc}(G)\}$, 
$\gamma_r(G) \leq \gamma_{tr}(G)$, $\gamma^{oc}(G) \leq \gamma_t^{oc}(G)$ and  $\gamma(G) \leq \min \{\gamma_r(G), \gamma^{oc}(G)\}$.

Not much work has been done on  finding relationships between
(a)  the value of a given domination parameter in the standard lexicographic product and that of its factors, and 
(b) the values of two given domination parameters in the standard lexicographic product. 
We list  those results that relate to the domination parameters above defined.

\begin{them}\label{known}
Let $G_1$ and $G_2$ be graphs with at least two vertices. 
\begin{itemize}
\item[(i)]  (Zhang et al. \cite{zlm}) If $\gamma(G_2)=1$, then $\gamma(G_1[G_2]) = \gamma(G_1)$. 
\item[(ii)]  (Zhang et al. \cite{zlm}) $\gamma_t(G_1[G_2]) = \gamma_t(G_1)$.
\item[(iii)] (Zhang et al. \cite{zlm})  If $\delta(G_1) \geq 1$, then $\gamma_r(G_1[G_2]) = \gamma_r(G_1)$.
\item[(iv)] (\u{S}umenjak et al.  \cite{spt}) If $G_1$ and $G_2$ are connected and $\gamma(G_2) \geq 2$, then 
                     $\gamma(G_1[G_2]) = \gamma_t(G_1[G_2]) = \gamma_t(G_1)$.    
\end{itemize}
\end{them}

In Section $2$ we obtain results on the parameters $\gamma$, $\gamma_t$, $\gamma_r$, $\gamma_{tr}$, $\gamma_p$, 
$\gamma^{oc}$ and  $\gamma_t^{oc}$   in a generalized lexicographic product. 
Some of them generalize or/and extend  those stated in the above theorem. 

To continue we need the following  definition. 
\begin{definition}\label{abwell}
Let $\mathcal{A}, \mathcal{B} \subseteq \mathcal{I}$. 
A graph $G$ is said to be well $\gamma_{(\mathcal{A}, \mathcal{B})}$-dominated if 
$\gamma_{(\mathcal{A}, \mathcal{B})} (G) = \Gamma_{(\mathcal{A}, \mathcal{B})} (G)$. 
\end{definition}

In a 1970 paper, Plummer \cite{p} introduced the notion of considering graphs in
which all maximal  independent sets have the same size; 
 he called a graph having this property a {\em well-covered graph}. Equivalently, 
a well-covered graph is one in which  the greedy algorithm for constructing 
independent sets yields always maximum independent sets.
Clearly well-covered graphs form the class of all well $i$-dominated graphs. 
Topp and Volkmann \cite{tv2} gave a  characterization   of 
well covered  generalized lexicographic product of graphs. 

Well $\gamma$-dominated graphs were introduced by Finbow  et al. \cite{fhn}. 
Obviously each well $\gamma$-dominated graph is well covered. 
Recall the characterization of the well-dominated nontrivial lexicographic product of two graphs,
which was recently obtained by  G\"{o}z\"{u}pek, Hujdurovi\'c and   Milani\v{c}		in \cite{GHM1}.

\begin{them} \label{wdf} \cite{GHM1} 
A nontrivial lexicographic product, $G[F]$, of a connected graph $G$ and 
a graph $H$ is well-dominated if and only if one of the following conditions holds:
\begin{itemize}
\item[(i)] $G$ is well-dominated and $F$ is complete, or
\item[(ii)] $G$ is complete and $F$ is well-dominated with $\gamma(F) = 2$.	
\end{itemize}
\end{them}		
		
In Section $3$ we present results on  well $\gamma_{(\mathcal{A}, \mathcal{B})}$-dominated graphs; 
in particular we characterize well $\gamma$-dominated  and well $\gamma_t$-dominated 
generalized lexicographic product of graphs.

  We conclude in Section 4 with some open problems.

\section{Seven domination parameters}

Recall that the equality $\gamma_t(G) = \gamma_t(G[F])$ was proven by   X. Zhang et al. \cite{zlm}.
The next theorem shows that the equality remains valid if we remove  $F$ by $\Phi$.

\begin{theorem}\label{t=t=}
  If $I$ is a $\gamma_t$-set of some $G$-layer of $G[\Phi]$, then $I$ is a $\gamma_t$-set of $G[\Phi]$. 
	In particular,  $\gamma_t(G) = \gamma_t(G[\Phi])$.  
\end{theorem}
\begin{proof}
  Let $U$ be a $G$-layer of $G[\Phi]$ and $I$ a $\gamma_t$-set of $\left\langle U \right\rangle$.
   By the definition of  a graph $G[\Phi]$ we immediately obtain that 
	$I$ is a total dominating set of $G[\Phi]$. 
   Since $\left\langle  U \right\rangle \simeq G$, $\gamma_t(G) \geq \gamma_t(G[\Phi])$. 
    Now if the equality holds, then  clearly $I$ is a $\gamma_t$-set of $G[\Phi]$. 
	
	For each $\gamma_t$-set $T$ of $G[\Phi]$ denote by $s_T$ 
the number  of all $i$ for which $T$ and $V(F_i)$ have at least two elements in common. 
Choose now a $\gamma_t$-set $R$ of $G[\Phi]$ so that $s_R$ is minimum.
Suppose $s_R\not=0$ and $x,y \in V(F_m) \cap R$ for some $m \in [n]$.
 Then $|R \cap V(F_l)|=0$  for all $F_l$'s such that $\textbf{l}\textbf{m} \in E(G)$. 
Consider now the set $R_1 = (R-\{y\}) \cup \{z\}$, where $z\in V(F_l)$. 
 Obviously $R_1$ is a $\gamma_t$-set $G[\Phi]$ with $s_R > s_{R_1}$,
 which contradicts   the choice of $R$. Thus $s_R=0$  and then there 
is a $G$-layer of  $G[\Phi]$,  say $H$, which contains $R$. Since clearly 
$R$ is a total dominating set of $\left\langle H \right\rangle$ and $\left\langle H \right\rangle \simeq G$, 
we obtain $\gamma_t(G) \leq \gamma_t(G[\Phi])$.
\end{proof}

\begin{theorem}\label{gamma=1=}
$\gamma(G) \leq \gamma(G[\Phi])$. 
The equality holds if and only if there is a $\gamma$-set $I$ of $G$ such that if $\textbf{i}_\textbf{j} \in I$
 is an isolated vertex of $\left\langle I \right\rangle$, then $\gamma(F_j) = 1$.
If $\gamma(G) = \gamma(G[\Phi])$, then  $|D \cap V(F_i)| \leq 1$,  $i=1,2,..,n$, 
 for each $\gamma$-set $D$ of $G[\Phi]$. 
\end{theorem}
\begin{proof}
Let $D$ be a $\gamma$-set of $G[\Phi]$  and let $F_{i_1},F_{i_2},..,F_{i_k}$ be all $F_j$'s 
each of which has  a common vertex with $D$. 
Choose a $G$-layer $U = \{u_1,u_2,..,u_n\}$ so that  $u_{i_s} \in V(F_{i_s})\cap D$ for $s=1,2,..,k$. 
Clearly $D_1 = \{u_{i_1}, u_{i_2},.., u_{i_k}\}$ is a dominating set  of $U$. 
Since  $\left\langle U \right\rangle \simeq G$, we have 
$\gamma(G) = \gamma(U) \leq |D_1| \leq |D| = \gamma(G[\Phi])$. 

Assume now that $\gamma(G) =  \gamma(G[\Phi]) =k$. 
Then $D_1$ is a $\gamma$-set of $U$ and $D_1=D$. 
This immediately implies $|D \cap V(F_i)| \leq 1$ for all $i \in [n]$. 
Since  $\left\langle U \right\rangle \simeq G$,  $I = \{\textbf{i}_\textbf{1}, \textbf{i}_\textbf{2},.., \textbf{i}_\textbf{k}\}$ is a  $\gamma$-set of $G$. 
If $\textbf{i}_\textbf{r}$ is an isolated vertex of $\left\langle I \right\rangle$, then  
$u_{i_r}$ is  an isolated vertex of $\left\langle D_1 \right\rangle$.
 Since $D=D_1$, $u_{i_r}$ is  an isolated vertex of $\left\langle D \right\rangle$ 
and therefore $u_{i_r}$ must dominate all vertices in $F_{i_r}$.

Let there be a $\gamma$-set   $I = \{\textbf{i}_\textbf{1}, \textbf{i}_\textbf{2},.., \textbf{i}_\textbf{k}\}$ of $G$ 
such that  for each isolated vertex $\textbf{i}_\textbf{s}$ in $\left\langle I \right\rangle$, $\gamma(F_{i_s}) = 1$. 
  But then   the set $R = \{x_{i_1}, x_{i_2},.., x_{i_k}\}$, where $x_{i_s} \in V(F_{i_s})$
	has maximum degree in $F_{i_s}$, $s=1,2,..,k$, is a dominating set of $G[\Phi]$. 
	Hence  $\gamma(G) \geq  \gamma(G[\Phi])$ and 	the required follows.
	\end{proof}

\begin{corollary}\label{g=1=}
If $\gamma(F_1) = \gamma(F_2) = ... = \gamma(F_n) = 1$,  
then  $\gamma(G) = \gamma(G[\Phi])$.  
\end{corollary}

By Theorem \ref{t=t=}, Theorem \ref{gamma=1=} and the well known inequalities $\gamma(G) \leq \gamma_t(G) \leq 2\gamma(G)$,  
we obtain the next inequality chain.
\begin{corollary}\label{ggt}
$\gamma(G) \leq \gamma(G[\Phi]) \leq \gamma_t(G[\Phi]) = \gamma_t(G) \leq 2\gamma(G) \leq 2\gamma(G[\Phi])$.
\end{corollary}

An immediate  consequence of this corollary is the following.

\begin{corollary}\label{g=gt=}
  If $\gamma(G) = \gamma_t(G)$, then  $\gamma(G[\Phi]) = \gamma_t(G[\Phi])$. 
	If $\gamma_t(G[\Phi]) =  2\gamma(G[\Phi])$, then $\gamma_t(G) =2\gamma(G)$.
\end{corollary}

 Now we concentrate on the case when all $F_i$'s have at least two vertices. 
We need the following key lemma for our purpose in this work.

\begin{lemma}\label{lexlemaanew}
Let  $\mu \in \{\gamma,\gamma_t\}$,          
$D$ be a $\mu$-set of $G[\Phi]$ and $|V(F_i)| \geq 2$ for all $i \in [n]$. 
Then the following assertions hold. 
\begin{itemize}
\item[(i)] $|D \cap V(F_i)| \leq 2$ for all  $i=1,2,..,n$.  
                    If $|D \cap V(F_s)|=2$ for some $s \in [n]$, then no vertex of $F_s$ 
										is adjacent to a vertex of $D-V(F_s)$ and $D \cap V(F_s)$ is a $\mu$-set  of $F_s$.  
										If $|D \cap V(F_i)| = |D \cap V(F_j)| = 2$, $i \not= j$,
										then 	the distance between any vertex of $F_i$ to any vertex of $F_j$ is at least three.  
\item[(ii)]	 Let  $R = \{i \mid  2=|D \cap V(F_i)|\} \not= \emptyset$. 
                    For all $i \in R$ let $D \cap V(F_i) = \{z_{i1}, z_{i2}\}$ 
										and $x_i  \in N(z_{i2}) -V(F_i)$. 
									Then the set $D^* =( D -\cup_{i \in R}\{z_{i2}\}) \cup \cup_{i \in R}\{x_i\}$ 
							     is a $\mu$-set of $G[\Phi]$ and $|D^* \cap V(F_r)| \leq 1$ for all $r \in [n]$.   
\item[(iii)] $G[\Phi]$ has a $\mu$-set $U$ such that  (a) $|U \cap V(F_i)| \leq 1$ for all  $i \in [n]$,  
                   (b) if $\mu = \gamma$, then $U$ is both a $\gamma_r$-set  and a $\gamma^{oc}$-set of $G[\Phi]$, and 										      
									 (c) if $\mu = \gamma_t$, then $U$ is both a  $\gamma_{tr}$-set and a $\gamma_t^{oc}$-set of $G[\Phi]$. 
\item[(iv)] Let $\mu = \gamma$ and $\gamma(F_i) \geq 2$ for all $i \in [n]$. 
                     Then a $\gamma$-set $U$ (see (iii))  is a $\nu$-set for all $\nu= \gamma_r, \gamma_t, \gamma_{tr}, \gamma^{oc}, \gamma_t^{oc}$. 
\end{itemize}
\end{lemma}										

\begin{proof}
(i)  Let for some $j \in [n]$ is fulfilled $D \cap V(F_j) = \{u_1,u_2,..,u_r\}$, where $r\geq 2$. 
      Then $N[\{u_1,w\}] \supseteq N[\{u_2,..,u_r\}]$ for any neighbor $w$ of $u_1$ which is outside  $V(F_j)$.   
			As $D$ is a $\mu$-set of $G[\Phi]$, we have $w \not\in D$ and 	$r=2$.
				Since $w$ was chosen arbitrarily,  $N[\{u_1,u_2\}] \cap D = \{u_1,u_2\}$. 
					But then $\{u_1,u_2\}$ is a $\mu$-set of $F_j$. 
		Let $|D \cap V(F_i)| = |D \cap V(F_s)| = 2$ for some $i,s \in [n], i \not= s$. 
		Suppose $z_1,z_2,z_3$ is a shortest path in $G[\Phi]$, where $z_1 \in V(F_i)$ and   $z_3 \in V(F_s)\cap D$. 
					Clearly $z_2 \not\in V(F_i) \cup V(F_j) \cup D$. 
					Then $D^\prime = (D-\{z_3\}) \cup \{z_2\}$ is a $\mu$-set and  $z_2$ is a 
					common neighbor of the vertices in $V(F_i) \cap D$,  a contradiction.
								
(ii) By (i) all 	$x_i$'s are paired distinct and outside $D$, 
        and no two of them belong to the same $F_j$, $j \in [n]$.  
       Hence $|D^*| = |D|$, $|D^* \cap V(F_r)| \leq 1$ for all $r \in [n]$
			and since $N[\{z_{i1}, z_{i2}\}] \subset N[\{z_{i1}, u_i\}]$ and $z_{i1}$ and $u_i$ are adjacent, 
			$D^*$ is a $\mu$-set of $G[\Phi]$. 
       
	(iii)  Define a $\mu$-set $U$ such that $U=D$ when $|D \cap V(F_i)| \leq 1$ for all  $i \in [n]$, and 
					 $U=D^*$ otherwise (see (ii)). Hence $|U \cap V(F_i)| \leq 1$
	         and  since $G$ is connected of order  $n \geq 2$ and $|V(F_i)| \geq 2$ for all $i \in [n]$, 
           a graph $\left\langle V(G[\Phi])-D \right\rangle$ is connected of order at least two. 
	
(iv) 	Since  $\gamma(F_i) \geq 2$ for all $i \in [n]$, $U$ is a total dominating set of $G[\Phi]$
        and since  $\gamma(G[\Phi]) \leq  \gamma_t(G[\Phi])$, 
					$U$ is a  $\gamma_t$-set of $G[\Phi]$.
					The required now immediately follows by (iii). 
\end{proof}

\begin{theorem}\label{two}
Let $|V(F_i)| \geq 2$ for all $i \in [n]$. Then
\begin{itemize} 
      \item[(i)] 		$\gamma(G[\Phi]) = \gamma_r(G[\Phi]) = \gamma^{oc}(G[\Phi])$, and 							
			\item[(ii)]   $\gamma_t(G) = \gamma_t(G[\Phi]) = \gamma_{tr}(G[\Phi]) = \gamma_t^{oc}(G[\Phi])$. 
\end{itemize}									
\end{theorem}
\begin{proof}
Immediately by Lemma \ref{lexlemaanew}(iii) and Theorem \ref{t=t=}. 
\end{proof}

Let $\mu, \nu \in \{\gamma, \gamma_r, \gamma^{oc}, \gamma_t, \gamma_{tr}, \gamma_t^{oc}\}$. 
If the sets of all $\mu$-sets and all $\nu$-sets of a graph $H$ coincide, then we say 
  $\mu(H)$ {\em strongly equal} to $\nu(H)$, written $\mu(H) \equiv \nu(H)$.

\begin{theorem}\label{three}
Let $|V(F_i)| \geq 3$ for all $i \in [n]$. Then
\begin{itemize} 
      \item[(i)] 		$\gamma(G[\Phi]) \equiv \gamma_r(G[\Phi]) \equiv \gamma^{oc}(G[\Phi])$, and 							
			\item[(ii)]   $\gamma_t(G[\Phi]) \equiv \gamma_{tr}(G[\Phi]) \equiv \gamma_t^{oc}(G[\Phi])$. 
\end{itemize}									
\end{theorem}
\begin{proof}
By Lemma \ref{lexlemaanew}(i) we know that for any $\mu$-set $D$, $\mu \in \{\gamma, \gamma_t\}$, 
of $G[\Phi]$ is fulfilled $|D \cap V(F_i)| \leq 2$ for all  $i=1,2,..,n$. 
Since $n \geq 2$ and $|V(F_i)| \geq 3$ for all $i \in [n]$, $D$ is both restrained and outer-connected. 
The rest immediately follows by the previous theorem.  
\end{proof}

\begin{theorem}\label{gamma=2}
If $\gamma(F_i) \geq 2$ for all $i \in [n]$, then
\[
\gamma(G[\Phi]) = \gamma_r(G[\Phi]) = \gamma_t(G[\Phi]) =  \gamma_{tr}(G[\Phi]) = 
        \gamma^{oc}(G[\Phi]) =  \gamma_t^{oc}(G[\Phi]).
\]
If $\gamma(F_i) \geq 3$ for all $i \in [n]$, then 
\[
\gamma(G[\Phi]) \equiv \gamma_r(G[\Phi]) \equiv \gamma_t(G[\Phi]) \equiv \gamma_{tr}(G[\Phi])  
        \equiv \gamma^{oc}(G[\Phi]) \equiv  \gamma_t^{oc}(G[\Phi]).
\]
\end{theorem}									
\begin{proof}
The first equality chain immediately follows by Lemma \ref{lexlemaanew}(iv). 
Assume now  $\gamma(F_i) \geq 3$ for all $i\in [n]$ and let $D$ be any $\gamma$-set of $G[\Phi]$. 
Lemma \ref{lexlemaanew}(i)  implies that $D$ must be total dominating. 
The required now follows by Theorem \ref{three}. 
\end{proof}

A dominating set $D$ of a graph  $G$ is {\em efficient dominating} if 
every vertex in $V(G)$ is dominated by exactly one vertex of $D$. 
Note that each efficient dominating set is a $\gamma$-set.

\begin{theorem}\label{eff}
Let  $D$ be a $\gamma$-set of $G[\Phi]$, $|V(F_i)| \geq 2$ and  $|D\cap V(F_i)| \not= 1$ for all $i \in [n]$. 
Let all $F_i$'s for which $|D\cap V(F_i)| \not= 0$ be $F_{i_1},F_{i_2},..,F_{i_s}$. 
Then $|D \cap V(F_{i_r})| = 2$ for all $r \in [s]$, $\{\textbf{i}_\textbf{1}, \textbf{i}_\textbf{2},.., \textbf{i}_\textbf{s}\}$
	is an efficient dominating set of $G$ and 
\begin{align}
2\gamma(G) 
                          & = \gamma_t(G) = \gamma(G[\Phi]) = \gamma_r(G[\Phi]) = \gamma_t(G[\Phi]) =  \gamma_{tr}(G[\Phi]) =  \gamma^{oc}(G[\Phi]) \nonumber \\
                          & =  \gamma_t^{oc}(G[\Phi]) = \gamma_p(G[\Phi]). \nonumber
\end{align}
\end{theorem}
\begin{proof}
By Lemma \ref{lexlemaanew}(i) the following hold: 
(a)  $|D \cap V(F_i)| \in \{0, 2\}$ for all  $i \in [n]$, 
(b) $|D\cap V(F_{i})| = 2$  if and only if $i \in \{i_1,i_2,..,i_s\}$,
(c)  $\{\textbf{i}_\textbf{1}, \textbf{i}_\textbf{2},..,\textbf{i}_\textbf{s}\}$ 	is an efficient dominating set of $G$ and 
(d)  $D\cap V(F_{i_r})$ is a $\gamma$-set of $F_{i_r}$ for all $r \in [s]$. 
Therefore $2\gamma(G) = \gamma(G[\Phi])$
and by  $\gamma_t(G[\Phi]) = \gamma_t(G) \leq 2\gamma(G) = \gamma(G[\Phi]) \leq  \gamma_t(G[\Phi])$, 
we have $\gamma_t(G[\Phi]) = \gamma(G[\Phi])$. 
In view of Theorem \ref{two}, it remains to prove that $\gamma(G[\Phi]) = \gamma_p(G[\Phi])$. 

Let $D \cap V(F_{i_r}) = \{z_{r1}, z_{r2}\}$ 	and $x_r  \in N(z_{r2}) -V(F_{i_r})$, $r \in [s]$. 
Since   $\{\textbf{i}_\textbf{1}, \textbf{i}_\textbf{2},..,\textbf{i}_\textbf{s}\}$ 	is an efficient dominating set of $G$, 
the set $D^* =( D -\cup_{r \in [s]}\{z_{r2}\}) \cup \cup_{r\in [s]}\{x_r\}$ 
is a dominating $(\mathcal{M}, \mathcal{C})$-set of $G[\Phi]$ of cardinality $|D^*| = |D| = \gamma(G[\Phi])$. 
Since each dominating $(\mathcal{M}, \mathcal{C})$-set is a dominating $(\mathcal{M}, \mathcal{I})$-set, 
 $\gamma_p(G[\Phi]) = \gamma(G[\Phi])$. 
\end{proof}

\begin{remark}\label{roc}
In the end of the proof of the above theorem we obtain that a set $D^*$ 
is a dominating $(\mathcal{M}, \mathcal{C})$-set of $G[\Phi]$. 
Hence under the assumptions of Theorem \ref{eff}, 
$2\gamma(G) = \gamma_{(\mathcal{M},\mathcal{C})}(G[\Phi])$. 
It is quite natural to call the  numbers $\gamma_{(\mathcal{M},\mathcal{C})}(G)$ and 
$\Gamma_{(\mathcal{M},\mathcal{C})}(G)$ the paired outer connected and 
upper paired outer connected domination numbers of a graph $G$. 
\end{remark}

\section{Well $\gamma_{(\mathcal{A}, \mathcal{B})}$-dominated graphs}

We begin with an  obvious but very useful observation.

\begin{observation}\label{eqchains}
Given a graph $G[\Phi]$ and properties $\mathcal{A}, \mathcal{B} \subseteq  \mathcal{I}$. 
Assume that $F_l$ has a dominating $(\mathcal{A}, \mathcal{B})$-set for all $l \in [n]$, 
and denote by $\mathscr{D}_{G[\Phi]}(\mathcal{A}, \mathcal{B} )$ the family of all subsets  $U$ of $V(G[\Phi])$  
such that $U = \cup_{r=1}^s D_{l_r}$, where 
$\{\textbf{l}_\textbf{1}, \textbf{l}_\textbf{2},..,\textbf{l}_\textbf{s}\}$
											is a maximal independent set of $G$ and $D_{l_i}$ is a minimal dominating  $(\mathcal{A}, \mathcal{B})$-set 
											of $F_{l_i}$, $i=1,2,..,s$. 
Let all elements of $\mathscr{D}_{G[\Phi]}$ be minimal dominating  $(\mathcal{A}, \mathcal{B})$-sets of  $G[\Phi]$. 
Then 			
\begin{align}
\label{eq:n1}
\gamma_{(\mathcal{A}, \mathcal{B})}(G[\Phi]) 	
                           & \leq \min\{|U| \mid U \in \mathscr{D}_{G[\Phi]} (\mathcal{A}, \mathcal{B})\} \nonumber \\
                          & =	 \min\{ \Sigma_{r=1}^s\gamma_{(\mathcal{A}, \mathcal{B})}(F_{i_r}) \mid
													\{\textbf{i}_\textbf{1}, \textbf{i}_\textbf{2},..,\textbf{i}_\textbf{s}\}  	 \in In(G)\}\nonumber \\
	                       	&\leq \min\{\Sigma_{r=1}^k\gamma_{(\mathcal{A}, \mathcal{B})} (F_{j_r}) \mid 
													\{\textbf{j}_\textbf{1}, \textbf{j}_\textbf{2},..,\textbf{j}_\textbf{k}\}   \mbox{\ is an $i$-set of  $G$}\} \\
													& \leq i(G)\max\{\gamma_{(\mathcal{A}, \mathcal{B})}(F_j) \mid j \in [n]\} \nonumber
\end{align}
and 
\begin{align}
\label{eq:n2}
\Gamma_{(\mathcal{A}, \mathcal{B})}(G[\Phi]) 
                     	    & \geq \max\{|U| \mid U \in \mathscr{D}_{G[\Phi]}(\mathcal{A}, \mathcal{B})\} \nonumber \\
                          & =	 \max\{ \Sigma_{r=1}^s\Gamma_{(\mathcal{A}, \mathcal{B})}(F_{i_r}) \mid 
													 \{\textbf{i}_\textbf{1}, \textbf{i}_\textbf{2},..,\textbf{i}_\textbf{s}\} 	 \in In(G)\}\nonumber \\
													&\geq \max\{\Sigma_{r=1}^s\Gamma_{(\mathcal{A}, \mathcal{B})} (F_{i_r}) \mid
													\{\textbf{j}_\textbf{1}, \textbf{j}_\textbf{2},..,\textbf{j}_\textbf{k}\}  \mbox{\ is a $\beta_0$-set of  $G$}\} \\
													& \geq \beta_0(G)\min\{\Gamma_{(\mathcal{A}, \mathcal{B})}(F_j) \mid j \in [n]\}, \nonumber
\end{align}
where $In(G)$ is the set of all  maximal independent sets of $G$.
\end{observation}

As a first consequence of  this observation 
a necessary condition for a generalized lexicographic product  of graphs to be 
well $\gamma_{(\mathcal{A}, \mathcal{B})}$-dominated follows. 

\begin{corollary}\label{necessary}
Under the conditions and notation of Observation \ref{eqchains}, 
if $G[\Phi]$ is well $\gamma_{(\mathcal{A}, \mathcal{B})}$-dominated, then 
all $F_i$'s are well $\gamma_{(\mathcal{A}, \mathcal{B})}$-dominated and 
$\gamma_{(\mathcal{A}, \mathcal{B})}(G[\Phi]) = \Sigma_{r=1}^s\gamma_{(\mathcal{A}, \mathcal{B})}(F_{i_r})$ 
for each $\{\textbf{i}_\textbf{1}, \textbf{i}_\textbf{2},..,\textbf{i}_\textbf{s}\} 	 \in In(G)$. 
\end{corollary}
\begin{proof}
By (\ref{eq:n1}) and (\ref{eq:n2}) we have 
$\gamma_{(\mathcal{A}, \mathcal{B})}(G[\Phi]) \leq \min\{ \Sigma_{r=1}^s\gamma_{(\mathcal{A}, \mathcal{B})}(F_{i_r}) \mid
													\{\textbf{i}_\textbf{1}, \textbf{i}_\textbf{2},..,\textbf{i}_\textbf{s}\} 	 \in In(G)\}$  and 
									$\max\{ \Sigma_{r=1}^s\Gamma_{(\mathcal{A}, \mathcal{B})}(F_{i_r}) \mid 
													\{\textbf{i}_\textbf{1}, \textbf{i}_\textbf{2},..,\textbf{i}_\textbf{s}\} 	 \in In(G)\} 
													\leq  \Gamma_{(\mathcal{A}, \mathcal{B})}(G[\Phi])$, respectively. 
Since 	$G[\Phi]$ is well $\gamma_{(\mathcal{A}, \mathcal{B})}$-dominated, 
 $\gamma_{(\mathcal{A}, \mathcal{B})}(G[\Phi]) = \min\{ \Sigma_{r=1}^s\gamma_{(\mathcal{A}, \mathcal{B})}(F_{i_r}) \mid
													\{\textbf{i}_\textbf{1}, \textbf{i}_\textbf{2},..,\textbf{i}_\textbf{s}\} 	 \in In(G)\} = 
	\max\{ \Sigma_{r=1}^s\Gamma_{(\mathcal{A}, \mathcal{B})}(F_{i_r}) \mid 
													\{\textbf{i}_\textbf{1}, \textbf{i}_\textbf{2},..,\textbf{i}_\textbf{s}\} 	 \in In(G)\} = 
													\Gamma_{(\mathcal{A}, \mathcal{B})}(G[\Phi])$.  
 	This equality chain immediately implies the required. 											
\end{proof}

\begin{corollary}\label{corchain}
Under the conditions and notation of Observation \ref{eqchains}, assume that 
$\gamma_{(\mathcal{A}, \mathcal{B})}(F_1) = \gamma_{(\mathcal{A}, \mathcal{B})}(F_2) =..=\gamma_{(\mathcal{A}, \mathcal{B})}(F_n)$ 
and $\Gamma_{(\mathcal{A}, \mathcal{B})}(F_1) = \Gamma_{(\mathcal{A}, \mathcal{B})}(F_2) =..=\Gamma_{(\mathcal{A}, \mathcal{B})}(F_n)$.
Then 
\begin{equation}\label{eq:n3}
\gamma_{(\mathcal{A}, \mathcal{B})}(G[\Phi]) \leq i(G)\gamma_{(\mathcal{A}, \mathcal{B})}(F_1)
 \leq \beta_0(G)\Gamma_{(\mathcal{A}, \mathcal{B})}(F_1) \leq \Gamma_{(\mathcal{A}, \mathcal{B})}(G[\Phi]). 
\end{equation}
If $G[\Phi]$ is well $\gamma_{(\mathcal{A}, \mathcal{B})}$-dominated, 
 then $G$ is well covered.
\end{corollary}
\begin{proof}
The middle inequality is a consequence of  $i(G) \leq \beta_0(G)$ and
 $\gamma_{(\mathcal{A}, \mathcal{B})}(G) \leq \Gamma_{(\mathcal{A}, \mathcal{B})}(G)$. 
The left and right inequalities follow immediately by (\ref{eq:n1}) and (\ref{eq:n2}), respectively. 
If $G[\Phi]$ is well $\gamma_{(\mathcal{A}, \mathcal{B})}$-dominated, 
then the inequality chain (\ref{eq:n3}) becomes equality chain 
implying $i(G) = \beta_0(G)$. 
\end{proof}

To formulate our next result,  we need the following domination parameters.
\begin{itemize}
\item [\rm (h)] $\gamma_{(\mathcal{F}, \mathcal{I})} (G)$ and  $\Gamma_{(\mathcal{F}, \mathcal{I})} (G)$ are 
                            the acyclic domination and upper acyclic domination numbers $\gamma_a(G)$ and $\Gamma_a(G)$ (\cite{hhr}), 
\item [\rm (i)]  $\gamma_{(\mathcal{S}_k, \mathcal{I})} (G)$ and  $\Gamma_{(\mathcal{S}_k, \mathcal{I})} (G)$ are 
                            the $k$-dependent domination and upper $k$-dependent domination numbers $\gamma^k(G)$ and $\Gamma^k(G)$ (\cite{fhhr}).  
\end{itemize}

\begin{remark}\label{appl}
Let  the pair $(\mathcal{A}, \mathcal{B})$  be one of 
$(\mathcal{I}, \mathcal{I})$, $(\mathcal{F}, \mathcal{I})$, $(\mathcal{S}_k, \mathcal{I})$,	$(\mathcal{I}, \mathcal{T})$, 
$(\mathcal{T}, \mathcal{I})$, $(\mathcal{T}, \mathcal{T})$ and $(\mathcal{M}, \mathcal{I})$. 
In addition if $(\mathcal{A}, \mathcal{B})$ is one of the last four pairs, let   $\delta(F_i) \geq 1$ for all $i \in [n]$. 
Then the assumptions of Observation \ref{eqchains} are fulfilled. 
Therefore the inequality chains (\ref{eq:n1}) and (\ref{eq:n2}) 
as well as  Corollary \ref{necessary} and Corollary \ref{corchain} are valid
for such a pair $(\mathcal{A}, \mathcal{B})$.
\end{remark}

The next result shows that the left and right inequalities in the chain (\ref{eq:n3}) become 
equalities in the case when $(\mathcal{A}, \mathcal{B}) = (\mathcal{S}_0, \mathcal{I})$ 
or equivalently, when 
$\gamma_{(\mathcal{A}, \mathcal{B})} = i$ 
and $\Gamma_{(\mathcal{A}, \mathcal{B})} = \beta_0$.

\begin{theorem}\label{ind} (\cite{nr} when $G[\Phi]=G[F]$) 
If $i(F_1) = i(F_2) =..= i(F_n)$, then  $i(G[\Phi]) = i(G)i(F_1)$. 
If $\beta_0(F_1) = \beta_0(F_2) =..= \beta_0(F_n)$, then  $\beta_0(G[\Phi]) = \beta_0(G)\beta_0(F_1)$. 
 \end{theorem}
\begin{proof}  By Remark \ref{appl}  and Corollary \ref{corchain} we know that  
                            $i(G[\Phi]) \leq i(G)i(F_1)$ and $\beta_0(G[\Phi]) \geq  \beta_0(G)\beta_0(F_1)$. 
              				      Let now $I$ be any maximal independent set  of  $G[\Phi]$ and let 
																	$F_{i_1}, F_{i_2},..,F_{i_s}$ be all $F_i$'s each of which has a vertex in $I$. 
																	Choose $u_{i_r} \in V(F_{i_r})$ arbitrarily, $r=1,2,..,s$, and consider any $G$-layer $H$ 
																	containing all vertices of  $U = \{u_{i_1}, u_{i_2},..,u_{i_s}\}$. 
																Clearly $U$ is a maximal independent set of $H \simeq G$;
																hence $\beta_0(G) \geq |U| \geq i(G)$.  
																It remains to note that obviously $I \cap V(F_{i_r})$ is  maximal independent set of $F_{i_r}$, $r \in [s]$. 
																Therefore  $i(G[\Phi]) \geq i(G)i(F_1)$ and $\beta_0(G[\Phi]) \leq  \beta_0(G)\beta_0(F_1)$.  
	\end{proof}

The following characterization of well covered  generalized lexicographic product of graphs
 is due to Topp and Volkmann \cite{tv2}.

\begin{them}\label{wellcovered}\cite{tv2}
The generalized lexicographic product $G[\Phi]$
is a well covered graph if and only if  all $F_i$'s are well covered and 
$\Sigma_{r=1}^s\beta_0(F_{i_r}) = \Sigma_{p=1}^l\beta_0(F_{j_p})$
for every two maximal independent sets 
$\{\textbf{i}_\textbf{1}, \textbf{i}_\textbf{2},..,\textbf{i}_\textbf{s}\}$ 
and 
$\{\textbf{j}_\textbf{1}, \textbf{j}_\textbf{2},..,\textbf{j}_\textbf{l}\}$
 of $G$. 
\end{them}

 Next we present a characterization of well $\gamma$-dominated  generalized lexicographic product of graphs.
For a graph $G[\Phi]$ and any minimal dominating set $R$ of $G$
let $I_R = \{\mathbf{i} \mid deg_{\left\langle R \right\rangle}(\mathbf{i}) = 0 \ \mbox{and} \ \gamma(F_i)\geq 2\}$. 
\begin{theorem}\label{neces}
Let $G[\Phi]$ be such that $|V(F_i)| \geq 2$ for all $i \in [n]$. 
 Then $G[\Phi]$ is  a well $\gamma$-dominated graph if and only if
 the following assertions hold. 
\begin{itemize}
\item[(i)]  $F_i$ is well $\gamma$-dominated with $\gamma(F_i) \leq 2$ for all $i \in [n]$. 
\item[(ii)]  there is a number $k$ such that for each minimal dominating set $R$ of $G$, $|R| + |I_R| = k$.  
\end{itemize}
\end{theorem}
\begin{proof}
$\Rightarrow$  
Let $D_j$ be a $\Gamma$-set of $F_j$ and $D_j^\prime$ a $\gamma$-set of $G- N[V(F_j)]$ for some $j \in [n]$.  
Then  $D_j \cup D_j^\prime$ is a minimal dominating set of $G[\Phi]$. 
Since $G[\Phi]$ is well dominated,  $D_j \cup D_j^\prime$ is a $\gamma$-set of $G[\Phi]$. 
Now using Lemma \ref{lexlemaanew}(i), we obtain that $|D_j| \leq 2$ and $D_j$ is a $\gamma$-set of $F_j$. 
Thus, (i) is satisfied. 

Let $R= \{\textbf{i}_\textbf{1}, \textbf{i}_\textbf{2},..,\textbf{i}_\textbf{s}\}$ be 
an arbitrary minimal dominating set of $G$ and let 
$U = \{u_1, u_2,.., u_s\}$ be a $G$-layer of $G[\Phi]$ such that 
$u_i$ belongs to some minimal dominating set of $F_i$, $i=1,2,..,n$.  
 Since $\left\langle U \right\rangle \simeq G$, 
$R_1 = \{u_{i_1}, u_{i_2},..,u_{i_s}\}$ is a  minimal dominating set of $U$. 
If $R_1$ is a dominating set of $G[\Phi]$, then clearly $R_1$ is a minimal dominating set of $G[\Phi]$; 
hence $I_R= \emptyset$ and $|R| + |I_R| = |R_1|$. Since  $G[\Phi]$ is well $\gamma$-dominated, 
$|R| + |I_R| = \gamma(G[\Phi]) = \Gamma(G([\Phi])$.
 Now let $I_R = \{\textbf{j}_\textbf{1}, \textbf{j}_\textbf{2},..,\textbf{j}_\textbf{l}\}$. 
Since $\gamma(F_{j_r}) = \Gamma(F_{j_r}) = 2$, for all $r \in [l]$ (by (i)), 
 there is $v_{j_r} \in V(F_{j_r})$ such that $\{u_{j_r} , v_{j_r} \}$ is a 
$\Gamma$-set of $F_{j_r}$, $r=1,2,..,l$. But then 
$R_1 \cup \{v_{j_1}, v_{j_2},..,v_{j_l}\}$ is a  minimal dominating set of $G[\Phi]$ 
and as $G[\Phi]$ is $\gamma$-well dominated, $\gamma(G([\Phi]) = \Gamma(G([\Phi]) = |R_1| + l = |R| + |I_R|$. 

$\Leftarrow$ 
   Let  $D$ be an arbitrary  minimal dominating set of $G[\Phi]$ and $F_{i_1},F_{i_2},..,F_{i_s}$ be all $F_i$'s 
   each of which  has an element in common with $D$.  Clearly 
   $R = \{\textbf{i}_\textbf{1}, \textbf{i}_\textbf{2},..,\textbf{i}_\textbf{s}\}$ 
   is a minimal dominating set of $G$. 
	Assume $D$ and $F_{i_r}$ have more than one element in common. 
	By (i), there are exactly two vertices belonging to both  $D$ and $F_{i_r}$. 
	But then for each $\textbf{j} \in N(\textbf{i}_\textbf{r})$,  $D \cap F_j$ is empty. 
	Therefore $\textbf{i}_\textbf{r} \in I_R$, which implies   $|D| = |R| + |I_R|$. 
Now by (ii), $|D| = k$ and since $D$ was chossen arbitrarily, $G[\Phi]$ is well $\gamma$-dominated.
 \end{proof}

By the proof of the above theorem we obtain the next result. 

\begin{corollary}\label{gPhi}
If $G[\Phi]$ is  well dominated  and $|V(F_i)| \geq 2$ for all $i \in [n]$, then 
for each minimal dominating set $R$ of $G$, $|R| + |I_R| = \gamma(G[\Phi])$.
\end{corollary}
 
 \begin{theorem}\label{wd2} (\cite{GHM1}  when $G[\Phi]=G[F]$) 
Let $G[\Phi]$ be such that $|V(F_i)| \geq 2$ for all $i \in [n]$, 
 $\gamma(F_1) = \gamma(F_2) =..=\gamma(F_n)$ and $\Gamma(F_1) = \Gamma(F_2) =..=\Gamma(F_n)$. 
Then  $G[\Phi]$ is well $\gamma$-dominated if and only if one of the following conditions holds:
\begin{itemize}
\item[(i)] $G$ is well-dominated and  all $F_i$'s are complete, or
\item[(ii)] $G$ is complete and $F_i$ is well $\gamma$-dominated with $\gamma(F_i) = 2$ for all $i \in [n]$. 
\end{itemize}
\end{theorem}
\begin{proof}
$\Rightarrow$ 
Assume first that $G[\Phi]$ is well $\gamma$-dominated. 
Using Remark \ref{appl} by Corollary \ref{necessary} and  Corollary \ref{corchain} we have
$i(G)= \beta_0(G)$ and $\Gamma(F_1) = \gamma(F_1)$. 
Let $I = \{\textbf{l}_\textbf{1}, \textbf{l}_\textbf{2},..,\textbf{l}_\textbf{s}\}$
be an arbitrary  $i$-set of $G$ and  $D_j$  an  arbitrary  $\gamma$-set  of $F_j$, $j=1,2,..,n$. 
Then clearly $D = \cup_{r=1}^s D_{l_r}$ is a $\gamma$-set of $G[\Phi]$.
Now by  Lemma \ref{lexlemaanew}(i) it follows that  $\gamma(F_1) \leq 2$. 
If  $\gamma(F_1) = 2$, then by Lemma \ref{lexlemaanew}(i) it follows that  
 each $i$-set of $G$ is efficient dominating. 
	This fact allow us to conclude that a graph $G$ is complete. 
											
	So, let $\gamma(F_1) = 1$. 	We already know that $\Gamma(F_1) = \gamma(F_1)$. 
                       Hence  all $\left\langle F_i \right\rangle$'s are complete.
										    Let $U$ be a $G$-layer of $G[\Phi]$, $R_1$ a $\gamma$-set of $U$ and 
												$R_2$ a $\Gamma$-set of $U$. Since all $\left\langle F_i \right\rangle$'s are complete, 
												both $R_1$ and $R_2$ are minimal dominating sets of  $G[\Phi]$ and since
												$G[\Phi]$ is well dominated,  
												$R_1$ and $R_2$ have the same cardinality. Thus, $U$ is well $\gamma$-dominated. 
												It remains to note that $G \simeq U$.
 
$\Leftarrow$ If (ii) is valid, then obviously $\gamma(G[\Phi]) = \Gamma(G[\Phi]) = 2$.
                           So,  suppose (i) is true and let $T_1$ and $T_2$ be different minimal dominating sets 
                            of    $G[\Phi]$.     Since all $\left\langle F_i \right\rangle$'s are complete, 
														there are two $G$-layers, say $U_1$ and $U_2$, which contain 		$T_1$ and $T_2$, 
														respectively.  
															Clearly $T_i$ is a minimal dominating set of $\left\langle U_i \right\rangle \simeq G$, $i=1,2$.
                            Since $G$ is well covered, $|T_1| = |T_2|$ and we are done. 
	\end{proof}

 A characterization of well $\gamma_t$-dominated  generalized lexicographic product of graphs follows.

\begin{theorem}\label{welltotal}
Given a graph  $G[\Phi]$ with $\delta(F_i) \geq 1$ for all $i \in [n]$. 
Then $G[\Phi]$ is  a well $\gamma_t$-dominated graph if and only if 
$G$ is complete and  for all $i \in [n]$, $F_i$ is well $\gamma_t$-dominated with $\gamma_t(F_i)=2$. 
Moreover, if  $G[\Phi]$ is  a well $\gamma_t$-dominated, then $\gamma_t(G[\Phi])=2$. 
\end{theorem}
\begin{proof}
$\Rightarrow$ 
Let $I = \{\textbf{l}_\textbf{1}, \textbf{l}_\textbf{2},..,\textbf{l}_\textbf{s}\}$
be an arbitrary  maximal independent set of $G$ and  $D_j$  an  arbitrary  $\Gamma_t$-set  of $F_j$, $j=1,2,..,n$. 
Then clearly $D = \cup_{r=1}^s D_{l_r}$ is a $\gamma_t$-set of $G[\Phi]$.
 Lemma \ref{lexlemaanew}(i) now implies that $I$  is an efficient dominating set of $G$ and 
  $D_{l_r}$ is a $\gamma_t$-set of $F_{l_r}$ with $\gamma_t(F_{l_r}) =2$ for all $r \in [s]$.
Since $I$ was chosen arbitrarily and each vertex of $G$ belongs to some  maximal independent set of $G$, 
	we can conclude that (a) all $F_i$'s are well $\gamma_t$-dominated graphs with $\gamma_t(F_i) =2$, and 
	(b) all  maximal independent sets of $G$ are efficient dominating. 
	The latter means that a graph $G$ is complete. 
		Finally, by Theorem \ref{t=t=}, $\gamma_t(G[\Phi]) = \gamma_t(G) = 2$.

$\Leftarrow$ Obviously each minimal  total dominating set of $G[\Phi]$ has cardinality $2$. 
 \end{proof}

Now we need  the following obvious but useful observation.

\begin{observation}\label{delta=0}
Given a graph  $G[\Phi]$ with $\delta(F_i) = 0$  and $|V(F_i)| \geq 2$ for all $i \in [n]$. 
Then a set $T$ is a minimal total dominating set of $G[\Phi]$ 
if and only if $T$  is a minimal total dominating set of some $G$-layer  of $G[\Phi]$. 
In particular, $\Gamma_t(G) = \Gamma_t(G[\Phi])$.
\end{observation}

\begin{theorem}\label{welltotal2}
Given a graph  $G[\Phi]$ with $\delta(F_i) = 0$ and $|V(F_i)| \geq 2$ for all $i \in [n]$. 
Then $G[\Phi]$ is  a well $\gamma_t$-dominated graph if and only if 
$G$ is well $\gamma_t$-total dominated. 
\end{theorem}
\begin{proof}
By Theorem \ref{t=t=} we have $\gamma_t(G[\Phi]) = \gamma_t(G)$, and 
by Observation \ref{delta=0} - $\Gamma_t(G) = \Gamma_t(G[\Phi])$.
Therefore $\gamma_t(G[\Phi]) = \Gamma_t(G[\Phi])$ if and only if $\gamma_t(G) = \Gamma_t(G)$.	
\end{proof}

In \cite{GHM1}  G\"{o}z\"{u}pek,  Hujdurovi\'c  and   Milani\v{c} posed the following problem.

\begin{prob}\label{m} \cite{GHM1} 
Characterize the nontrivial lexicographic product graphs that are well $\gamma_t$-dominated.
\end{prob}

The previous two theorems together give us the following characterization result.

 \begin{theorem}\label{wtd}
Let $G[F]$ be such that $|V(G)|, |V(F)| \geq 2$ and $G$ connected.  
Then  $G[F]$ is well-$\gamma_t$-dominated if and only if one of the following conditions holds:
\begin{itemize}
\item[(i)] $G$ is complete and  $F$ is well $\gamma_t$-dominated with $\gamma_t(F)=2$.
\item[(ii)] $G$ is well $\gamma_t$-dominated and $\delta(F) = 0$. 
\end{itemize}
\end{theorem}

\section{Open problems}
We conclude the paper by listing some interesting problems and directions for further research.

\begin{itemize}
\item[$\bullet$]  Find results on well $\mu$-dominated graphs, where $\mu$ is at least one of 
																	$\gamma_r, \gamma^{oc}, \gamma_{tr}, \gamma_t^{oc}, \gamma_a, \gamma_p, \gamma^k, k \geq 1$. 
                                 In particular, characterize the generalized lexicographic product graphs that are
                                  well $\mu$-dominated.
\end{itemize}

\begin{itemize}
\item[$\bullet$]   Characterize/describe those graphs $G$ having an efficient dominating set 
                                  of cardinality $\gamma_t(G)/2$  (see Theorem \ref{eff}). Such graphs are 
																	all circulants $C(4k+2; \{1,2,..,k\} \cup \{n-1,n-2,..,n-k\}$, $k \geq 1$.
  \end{itemize}

\begin{itemize}
\item[$\bullet$] 
                                Find results on dominating $(\mathcal{M}, \mathcal{C})$-sets (see Remark \ref{roc}).
  \end{itemize}

\begin{itemize}
\item[$\bullet$]                                
															Characterize/describe those generalized lexicographic product of graphs $G[\Phi]$ for which 
															at least one of the following holds: 
															$\gamma_{(\mathcal{A}, \mathcal{B})}(G[\Phi]) = i(G)\gamma_{(\mathcal{A}, \mathcal{B})}(F_1)$ 
														 and 
														 $\beta_0(G)\Gamma_{(\mathcal{A}, \mathcal{B})}(F_1) = \Gamma_{(\mathcal{A}, \mathcal{B})}(G[\Phi])$ 
														(see Corollary \ref{corchain}).
  \end{itemize}

\end{document}